\documentclass[12pt]{amsart}
\usepackage{amssymb, graphicx}
\usepackage[mathscr]{eucal}
\usepackage{amsmath,amscd}
\usepackage{rotating}
\usepackage{epstopdf}
\usepackage{diagbox}
\usepackage{multirow}
\usepackage{amssymb}
\usepackage{adjustbox}
\usepackage{float}
\usepackage{amsfonts}
\usepackage[final]{pdfpages}
\usepackage[dvipsnames]{xcolor}
\usepackage{ulem}
\usepackage{bbm}
\usepackage{float}
\usepackage{url}

\newcommand{\Z}{\mathbbm{Z}}

\theoremstyle{plain}
\newtheorem{prop}{Proposition}[section]
\newtheorem{coro}[prop]{Corollary}

\newtheorem{lemm}[prop]{Lemma}
\newtheorem{ques}[prop]{Question}
\newtheorem{thm}[prop]{Theorem}

\theoremstyle{definition}
\newtheorem{defn}[prop]{Definition}

\newtheorem{exmp}[prop]{Example}
\newtheorem{rem}[prop]{Remark}

\DeclareMathOperator{\mdet}{mdet}

\def\mcg#1;#2{\Gamma_{#1,#2}}
\def\fg#1;#2{\Pi_{#1,#2}}
\def\tb#1;#2{\mathscr{K}_{\frac{#1}{#2}}}

\begin{document}

\title[On Alternateness of the Jones-Kauffman Polynomial of Virtual Links]
{On Alternateness of the Jones-Kauffman Polynomial of Virtual Links}


\author{Hamid Abchir}
\address{Fundamental and Applied Mathematics Laboratory, Hassan II University, Casablanca, Morocco}
\email{hamid.abchir@univh2c.ma}

\author{Khaled Qazaqzeh}
\address{Department of Mathematics, Faculty of Science, Yarmouk University, Irbed, Jordan, 21163}
\email{qazaqzeh@yu.edu.jo}

\author{Mohammed Sabak}
\address{Fundamental and Applied Mathematics Laboratory, Hassan II University, Casablanca, Morocco}
\email{mohammed.sabak-etu@etu.univh2c.ma}



\begin{abstract}
  We study alternateness of the Jones-Kauffman polynomial of virtual links. In particular, we identify a set of virtual links with alternating Jones-Kauffman polynomial that will be called the set of mock-alternating virtual links. Then, we explain how every mock-alternating virtual link can give rise to an infinite family of mock-alternating virtual links using the twisting technique introduced in \cite{CK} and later generalized in \cite{CKa,QCQ}. Finally, we give a connected state expansion of the mock-determinant for checkerboard-colorable links.
\end{abstract}

\maketitle
\section{introduction}
The independent work of Thistlethwaite in \cite{Th} and Murasugi in \cite{M} demonstrated that the Jones polynomial of a non-split classical alternating link is indeed alternating. This finding was later extended to the set of quasi-alternating links by Ozsváth and Szabó in \cite{OS}. While the concept of alternating links was readily adapted to virtual links, it is unfortunate that the Jones-Kauffman polynomial, which serves as the natural extension of the Jones polynomial to virtual links, does not necessarily retain the alternating property for alternating virtual links. This raises a pertinent question: Is it possible to identify or characterize a set of virtual links with alternating Jones-Kauffman polynomial? 

In this paper, we identify a set of virtual links that is called the set of mock-alternating links. We show that each link in this set has  alternating Jones-Kauffman polynomial and this provides an easy obstruction for a link to be mock-alternating.  To achieve this, we initially define the mock-determinant as a virtual link invariant being the fundamental tool in establishing the set of mock-alternating virtual links. The mock-determinant is defined as the $L_{1}$-norm of the specialization of the Jones-Kauffman polynomial at the eighth root of unity $e^{\frac{i\pi}{4}}$. This formulation can be viewed as an extension of the determinant from classical links to virtual ones. We give some properties and derive a formula for the mock-determinant in terms of connected states for checkerboard colorable links using the connected state expansion of the Jones--Kauffman polynomial introduced in \cite{abchir}. As noted in \cite[Section\,1.5]{GPV}, there may exist multiple invariants that extend the determinant as a numerical link invariant for classical links. Indeed, other extensions of the determinant are discussed in \cite{BK,imlee2010} in the case of checkerboard colorable virtual links. However, we note that these extensions do not coincide for some virtual links. Consequently, the findings presented in this paper are based upon the specific definition chosen for the extension of the determinant.

\section{Jones-Kauffman polynomial of virtual links}

The Jones-Kauffman polynomial $f_L(t)$ is an invariant of oriented virtual links that was first introduced in \cite{Ka} and motivated by the work in \cite{J,Ka2}. It is a Laurent polynomial with integral coefficients that can be defined in terms of the Kauffman bracket. We briefly recall the definition of this polynomial as follows:

\begin{defn}
The Kauffman bracket polynomial is a function from the set of unoriented virtual link diagrams in the oriented plane to the ring of Laurent polynomials with integer coefficients in an indeterminate $A$. It maps a virtual link diagram $D$ to $\left\langle D\right\rangle\in \mathbb Z[A^{-1},A]$ and it is defined by the following relations:
\begin{enumerate}
\item $\left\langle \bigcirc \right\rangle=1$,
\item $\left\langle \bigcirc \cup D\right\rangle=(-A^{-2}-A^2)\left\langle D
\right\rangle$,
\item $\left\langle D\right\rangle=A\left\langle D_0\right\rangle+A^{-1}\left\langle
D_1\right\rangle$,
\end{enumerate}
where $\bigcirc$, in relation 2,  denotes a trivial circle  disjoint from the rest of the virtual link diagram, and  $D,D_0, \text{and } D_1$  represent three unoriented virtual link diagrams which are identical everywhere  except in a small region where they are as indicated in  Figure \ref{figure}.
\end{defn}
For an oriented virtual link diagram $D$, we let $x(D)$ denotes the number of negative crossings and  $y(D)$ denotes the number of positive crossings in $D$ according to the scheme in Figure \ref{Diagram1}. The writhe of the virtual link diagram $D$ is defined to be the integer $w(D) = y(D) - x(D)$.
\begin{defn}
The Jones-Kauffman polynomial $f_{L}(A)$ of the oriented virtual link $L$ is the Laurent polynomial in $A$ with integer coefficients defined by
\begin{equation*}
f_{L}(A)=(-A)^{-3w(D)}\left\langle D \right\rangle \in \mathbb Z[A,A^{-1}],
\end{equation*}
where $D$ is an oriented diagram for $L$, and $\left\langle D \right\rangle $ is the Kauffman bracket of the unoriented virtual link obtained from $D$ by ignoring  the orientation.
\end{defn}
\vspace{-0.2cm}
\begin{figure}[h]
	\centering
		\reflectbox{\includegraphics[scale=0.11]{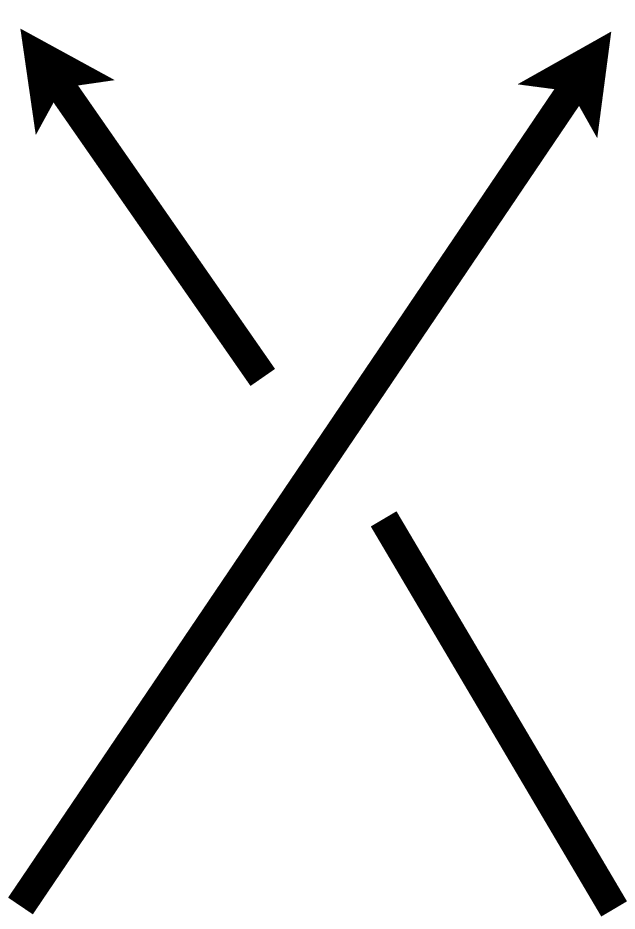}}\hspace{1.1cm}
		\includegraphics[scale=0.11]{Diagram1}
		\vspace{-0.2cm}
	\caption{Negative and positive crossings respectively.}
	\label{Diagram1}
\end{figure}
\begin{rem}
In this paper, we restrict ourselves to the case where the crossing $c$ in the link diagram of $L$ is as illustrated in Figure \ref{figure} and in this case such a crossing is said to be of type I. Similar results can be obtained for the other type of crossing by rotating the crossing by $\dfrac{\pi}{2}$.
\end{rem}

From now on, we let $c$ be a crossing of the oriented virtual link diagram $D$ that represents the oriented virtual link $L$. Also, we let $D_{0}$ and $D_{1}$ be the link diagrams obtained by smoothing the diagram $D$ at the crossing $c$ with corresponding virtual links $L_{0}$ and $L_{1}$, respectively.  We point out that the orientation on $D$ induces an orientation for one of the two diagrams $D_{0}$ or $D_{1}$ and the second one can be equipped with any orientation. Therefore, we can always assume that $D_{0}$ and $D_{1}$ are oriented diagrams for the oriented virtual links $L_{0}$ and $L_{1}$ if necessary. 
\begin{rem}\label{simple2} 
\begin{enumerate}
\item If the crossing is positive of type I, then we have $x(D_{0}) = x(D), y(D_{0}) = y(D) - 1, x(D_{1}) = x(D) + e$ and
$y(D_{1}) = y(D)-e-1$. Therefore, $w(D_{0}) =w(D)-1$ and $w(D_{1}) = w(D) -2e-1$,
where $e$ denotes the difference between the number of negative crossings in
$D_{1}$ and the number of negative crossings in $D$.
\item If the crossing is negative of type I, then we have $x(D_{1}) = x(D)-1, y(D_{1}) = y(D), x(D_{0}) = x(D) + e-1$ and
$y(D_{0}) = y(D)-e$. Therefore, $w(D_{1}) =w(D)+1$ and $w(D_{0}) = w(D) - 2e+1$,
where $e$ denotes the difference between the number of negative crossings in
$D_{0}$ and the number of negative crossings in $D$.
\end{enumerate}
\end{rem}

\begin{lemm}\label{jonespolynomial}
The Jones-Kauffman polynomial of the link $L$ with a fixed crossing $c$ satisfies one of the following skein relations:
\begin{enumerate}
\item If $c$ is a positive crossing of type I, then
$ f_{L}(A) = -A^{-2}f_{L_{0}}(A)-A^{-6e-4}f_{L_{1}}(A)$.
\item If $c$ is a negative crossing and of type I, then
$ f_{L}(A) = -A^{-6e+4}f_{L_{0}}(A)-A^{2}f_{L_{1}}(A)$,
\end{enumerate}
where $e$ is defined in Remark \ref{simple2}.
\end{lemm}
We recall \cite[Theorem\, 3]{KM} that gives a splitting of the Jones-Kauffman polynomial. This splitting is crucial in defining the mock-determinant of any virtual link that will be used in defining the set of mock-alternating virtual links. 
\begin{thm}\label{splitting}\cite[Theorem\, 3]{KM}
Let $L$ be a $\mu$-component virtual link. Then the Jones-Kauffman polynomial $f_{L}(A)$ can be written as $f_{L}(A) = \Phi_{L}(A) + \Psi_{L}(A)$, where $\Phi_{L}(A) \in \mathbb{Z}[A^{4}, A^{-4}]\cdot A^{2(\mu-1)}$ and $\Psi_{L}(A) \in \mathbb{Z}[A^{4}, A^{-4}]\cdot A^{2\mu}$. 
\end{thm}

\begin{rem}\label{simple}
To simplify the arguments and to avoid any confusion, we write always $f_{L}(A) = \Phi_{L}(A) + \Psi_{L}(A)$ with $ \Phi_{L}(A) \in \mathbb{Z}[A^{4}, A^{-4}]$ and $\Psi_{L}(A) \in \mathbb{Z}[A^{4}, A^{-4}]\cdot A^{2}$ for the link $L$.
\end{rem}

\section{The mock-Determinant}
In this section, we define the mock-determinant of virtual links which is the main tool used in the definition of the set of mock-alternating links. We also study some of its properties and provide a connected state expansion for it in the case of checkerboard colorable links.
\begin{defn}\label{determinant}
The mock-determinant of the virtual link $L$ denoted by $\mdet(L)$ is defined to be the $L_{1}$-norm of the complex number $f_{L}(\xi)$ that is $ \mdet(L) = ||f_{L}(\xi)||_{1} = \left|\Re\left(f_{L}(\xi)\right)\right| + \left|\Im(f_{L}(\xi))\right| = |\Phi_{L}(\xi)| + |\Psi_{L}(\xi)|$, where $\xi = \frac{1+i}{\sqrt{2}}$ and $\Re(x),$ and $\Im(x)$ are the real and the imaginary parts of the complex number $x$, respectively.
\end{defn}

\begin{rem}
According to Theorem \ref{splitting} and Remark \ref{simple}, we note that $\Phi_{L}\left(\xi\right)$ and $\Psi_{L}\left(\xi\right)$ are pure real and pure imaginary, respectively. This definition extends the definition of the determinant for classical links as a consequence of the facts that $\Phi_{L}(A) =0 $ or $\Psi_{L}(A) =0 $ in Theorem \ref{splitting} and  $\det(L) = |f_{L}(\xi)|=||f_{L}(\xi)||_{1}=\mdet(L)$ for any classical link $L$. 
\end{rem}

\begin{prop}\label{split}
If $L$ is a split link, then $\mdet(L) = 0$.
\end{prop}
\begin{proof}
If $L$ splits into two sublinks $L_{1}$ and $L_{2}$, then it is not too hard to see that $f_{L}(A) = (-A^{-2}-A^{2})f_{L_{1}}(A)f_{L_{2}}(A)$. As a consequence of this, we obtain $\Phi_{L}(A) =  (-A^{-2}-A^{2})g(A)$ and $\Psi_{L}(A) =  (-A^{-2}-A^{2})h(A)$ for some polynomials $g(A)\in \mathbb{Z}[A^{4},A^{-4}]\cdot A^{2}$ and $h(A)\in \mathbb{Z}[A^{4},A^{-4}]$.  Now $\mdet(L) = ||f_{L}(\xi)||_{1} = \left|\Phi_{L}\left(\xi\right)\right| + \left|\Psi_{L}\left(\xi\right)\right| = 0 + 0 = 0$.
\end{proof}

\begin{prop}
If $L$ is the connected sum of the two links $L_{1}$ and $L_{2}$, then $\mdet(L) \leq \mdet(L_{1}) \mdet(L_{2})$ and the equality holds if one of the two links is checkerboard colorable virtual link.
\end{prop}

\begin{proof}
If the link $L$ is the connected sum of the links $L_{1}$ and $L_{2}$ that is $L = L_{1}\#L_{2}$, then we know that $f_{L}(A) = f_{L_{1}}(A) f_{L_{2}}(A)$. In such a case, we obtain $\Phi_{L}(A) = \Phi_{L_{1}}(A)\Phi_{L_{2}}(A) + \Psi_{L_{1}}(A)\Psi_{L_{2}}(A)$ and  $\Psi_{L}(A) = \Psi_{L_{1}}(A)\Phi_{L_{2}}(A) + \Phi_{L_{1}}(A)\Psi_{L_{2}}(A)$. Now we have 
\begin{align*}
\mdet(L) & =  ||f_{L}(\xi)||_{1} = \left|\Phi_{L}\left(\xi\right)\right| + \left|\Psi_{L}\left(\xi\right)\right| \\ &
 =  \left|\Phi_{L_{1}}\left(\xi\right)\Phi_{L_{2}}\left(\xi\right)+ \Psi_{L_{1}}\left(\xi\right)\Psi_{L_{2}}\left(\xi\right)\right| +\left|\Psi_{L_{1}}\left(\xi\right)\Phi_{L_{2}}\left(\xi\right)+ \Phi_{L_{1}}\left(\xi\right)\Psi_{L_{2}}\left(\xi\right)\right| \\ &
\leq  \left|\Phi_{L_{1}}\left(\xi\right)\Phi_{L_{2}}\left(\xi\right)\right|+\left|\Psi_{L_{1}}\left(\xi\right)\Psi_{L_{2}}\left(\xi\right)\right| +\left|\Psi_{L_{1}}\left(\xi\right)\Phi_{L_{2}}\left(\xi\right)\right|+\left| \Phi_{L_{1}}\left(\xi\right)\Psi_{L_{2}}\left(\xi\right)\right| \\ &
=  \left|\Phi_{L_{1}}\left(\xi\right)\right|\left|\Phi_{L_{2}}\left(\xi\right)\right|+\left|\Psi_{L_{1}}\left(\xi\right)\right|\left|\Psi_{L_{2}}\left(\xi\right)\right| +\left|\Psi_{L_{1}}\left(\xi\right)\right|\left|\Phi_{L_{2}}\left(\xi\right)\right| + \left| \Phi_{L_{1}}\left(\xi\right)\right|\left|\Psi_{L_{2}}\left(\xi\right)\right| \\ & 
=  \left(\left|\Phi_{L_{1}}(\xi)\right|+ \left|\Psi_{L_{1}}(\xi)\right| \right) \left(\left|\Phi_{L_{2}}(\xi)\right| + \left|\Psi_{L_{2}}(\xi)\right|\right) \\& = ||f_{L_{1}}(\xi)||_{1}||f_{L_{2}}(\xi)||_{1} = \mdet(L_{1})\mdet(L_{2}).
\end{align*}
In the case that one of the two components is checkerboard colorable virtual link, then the triangle inequality becomes an equality because one of the four polynomials $\Phi_{L_{1}}(A), \Phi_{L_{2}}(A),$ $\Psi_{L_{1}}(A)$ or $\Psi_{L_{2}}(A)$ will be zero according to the result in \cite[Theorem\,1]{K}.
\end{proof}

\begin{prop}
If the number of components of $L$ is odd or even then $\mdet(L)$ is odd or even respectively.
\end{prop}

\begin{proof}
The Jones-Kauffman polynomial satisfies the well-known skein relation: 
$A^{4}f_{L_{+}}(A) - A^{-4} f_{L_{-}}(A) = (A^{-2}-A^{2})f_{L_{0}}(A)$, where $L_{-}$ and $L_{+}$ are the link diagrams given in Figure \ref{Diagram1} respectively.
Now evaluating such skein relation at $A= \xi$ and noting that the real and the imaginary parts of the complex number $(\xi^{-2} -\xi^{2})f_{L_{0}}(\xi)$ are even imply that the real and the imaginary parts of the  $f_{L_{+}}(\xi)$ and $f_{L_{-}}(\xi)$ have the same parity modulo two. Thus if the result holds for the link $L_{-}$ and $L_{0}$, then it easy to see that it follows for the link $L_{+}$ by induction on crossing changes to yield a link of crossing number zero. 
\end{proof}

Now we give a connected state expansion of the mock-determinant for checkerboard colorable links. Before we do so we recall some notions and terminologies first for this section.

We let $D$ be a checkerboard colorable virtual link diagram of the virtual link $L$ and $c_1,\dots ,c_n$ be its crossings arbitrarily ordered. We consider the set $\mathcal{S}_D$ of connected states of $D$ and we let $s$ be an element of $\mathcal{S}_D$. In the vicinity of each deleted crossing $c_i$, we put a marker which is a segment connecting two points on the arcs of the smoothing of $c_i$. Then we label each marker by $A$, $a$, $B$, or $b$ according to the type of the crossing and the checkerboard coloring of the diagram $D$ at $c_i$ as shown by the rules in Figure \ref{fig9}. An example of a labeling of a connected state of a virtual checkerboard virtual link diagram is depicted in Figure \ref{diagramplusstate}.

We denote by $A_s,a_s,B_s,$ and $b_s$  the number of $A,a,B,$ and $b$ markers in the state $s$, respectively. A crossing $c_i$ of $D$ is \textit{internal} to $s$ if the $i$-th marker is either an $A$-marker or $B$-marker otherwise $c_i$ is \textit{external} to $s$.
 We say that two distinct crossings $c_i$ and $c_j$ \textit{interlace} in $s$ if one visits alternately the smoothing arcs of $c_i$ and $c_j$ when traversing the circle generated by the state $s$. The labeled state on the right of Figure \ref{diagramplusstate} shows that the crossings $c_3$ and $c_4$ interlace, but $c_3$ does not interlace with either $c_1$ or $c_2$.

For a crossing $c_i$ of $D$, we denote the subset of $\left\lbrace 1,2, \hdots ,n \right\rbrace$ consisting of the indices of crossings interlacing with $c_i$ in $s$ by $\mathcal{C}_i(s)$. We denote the complement of $\mathcal{C}_{i}(s)$ in $\left\lbrace 1, \hdots , n \right\rbrace$ by $\overline{\mathcal{C}}_{i}(s)$.
The crossing $c_i$ is \textit{live} with respect to $s$ if $\mathcal{C}_i(s) \subset \left\lbrace i+1, \hdots ,n \right\rbrace$ otherwise $c_i$ is \textit{dead} with respect to $s$. Being internal or external and live or dead determine the \textit{state} of the crossing $c_i$ with respect to $s$.
 \begin{figure}
\centering
\includegraphics[scale=0.15]{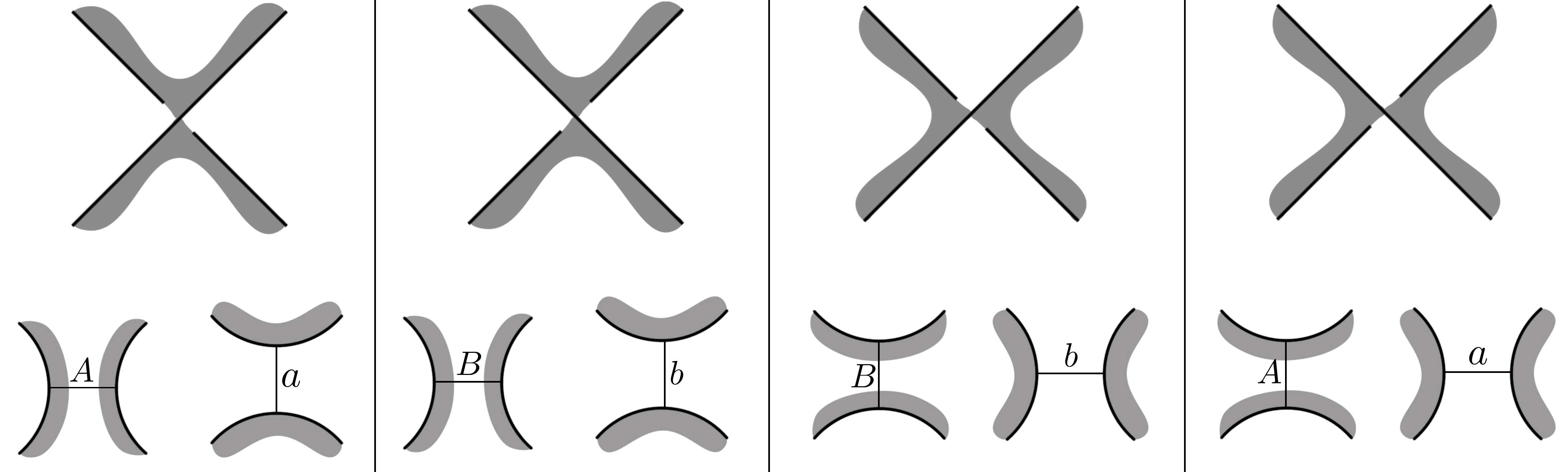}   
\caption{The marker corresponding to the smoothed crossing and the rule of its la beling.}
\label{fig9}
\end{figure}

\begin{figure}
\centering
\includegraphics[scale=0.15]{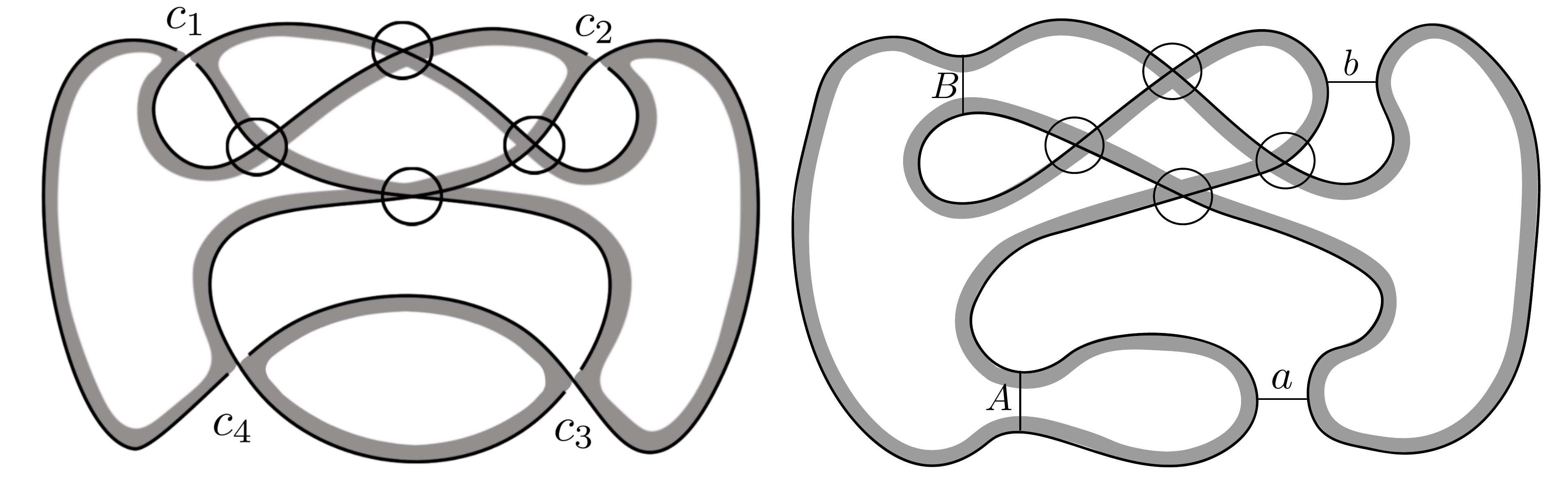}   
\caption{A checkerboard virtual link diagram and the labeling of one of its connected states.}
\label{diagramplusstate}
\end{figure}

The activity word $w(s)$ of $s$ is a word of $n$ letters in the alphabet $\{L, D, l, d, \overline{L}, \overline{D}, \overline{l}, \overline{d}\}$ such that the $i$-th activity letter $w_i(s)$ corresponding to the crossing $c_i$ is chosen according to Table (\ref{table1}). Once the activity word $w(s)$ is established, we assign to each letter $w_i (s)$ a monomial called the weight of the crossing $c_i$ with respect to $s$ and denoted by $\mu_i(s)\in \Z[q^{-1},q]$ according to Table (\ref{table2}). The weight of the connected state $s$ is the monomial $\mu(s):=\prod_{1\le i\le n}\mu_i(s)$. Now we define the Laurent polynomial $X_D$ for $D$ as follows:
 $$X_D(q)=\sum_{s\in\mathcal{S}_D}\mu(s).$$
 \renewcommand{\arraystretch}{1.3} 
 \begin{table}[ht]
\begin{center}
\begin{tabular}{|ll|l|l|}
\hline
\multicolumn{2}{|l|}{State of $c_i$ with respect to $s$}                   & The label of the $i$-th marker & The activity letter $w_i(s)$ \\ \hline
\multicolumn{1}{|l|}{\multirow{4}{*}{Internal}} & \multirow{2}{*}{Live} & $A$   &  {$L$} \\  \cline{3-4} 
\multicolumn{1}{|l|}{}                          &                       & $B$   & {$\overline{L}$}    \\   \cline{2-4} 
\multicolumn{1}{|l|}{}                          & \multirow{2}{*}{Dead} & $A$   & {$D$}   \\   \cline{3-4} 
\multicolumn{1}{|l|}{}                          &                       & $B$   &  {$\overline{D}$}   \\  \hline
\multicolumn{1}{|l|}{\multirow{4}{*}{External}} & \multirow{2}{*}{Live} & $a$   & {$l$}    \\  \cline{3-4} 
\multicolumn{1}{|l|}{}                          &                       & $b$ &  {$\overline{l}$}    \\  \cline{2-4} 
\multicolumn{1}{|l|}{}                          & \multirow{2}{*}{Dead} & $a$ & {$d$}   \\ \cline{3-4} 
\multicolumn{1}{|l|}{}                          &                       & $b$ & {$\overline{d}$}   \\ \hline
\end{tabular}
\vspace{0.2cm}
\caption{The activity of the crossing $c_{i}$ with respect to the state $s$ and the corresponding letter $w_i(s)$.}
\label{table1}
\end{center}
\end{table}

\begin{table}[ht]
\centering
\begin{tabular}{|c|c|c|c|c|c|c|c|c|} 
\hline 
$w_i(s)$ & $L$ & $D$ & $l$ & $d$ &  $\overline{L}$ & $\overline{D}$ &  $\overline{l}$ & $\overline{d}$ \\
\hline
$\mu_i(s)$ & $-q^{-3}$ & $q$ & $-q^3$ & $q^{-1}$ & $-q^3$ & $q^{-1}$ & $-q^{-3}$ & $q$ \\
\hline
\end{tabular}
\vspace{0.2cm}
\caption{The weight of the letter  $w_{i}(s)$ in the word $w(s)$.}
\label{table2}
\end{table}

 \begin{rem}
 The above formula gives a connected state expansion of the Jones-Kauffman polynomial of virtual links according to  \cite[Corollary\,3.4]{abchir}.
 \end{rem}
\begin{prop}\label{newprop1}
If $D$ is a connected checkerboard colorable virtual link diagram, then ${\rm mdet}(D)=\left|\displaystyle\sum_{s\in \mathcal{S}_D}i^{-B_s-a_s}\right| = \left|\displaystyle\sum_{s\in \mathcal{S}_D}i^{A_s+b_s}\right|$.
\end{prop}
\begin{proof}
For $\varepsilon \in \left\lbrace L,\overline{L},l,\overline{l},D,\overline{D},d,\overline{d} \right\rbrace$, we denote by $\varepsilon_s$ the number of occurrences of the letter $\varepsilon$ in the activity word of a connected state $s$. By \cite[Corollary\,3.4]{abchir} with the normalization $A=q$, we have
  \begin{align*}
X_D(q) & =\sum_{s \in \mathcal{S}_D} (-q^{-3})^{L_s+\overline{l}_s} (-q^3)^{\overline{L}_s+l_s} q^{D_s + \overline{d}_s} (q^{-1})^{\overline{D}_s+d_s} \\
& = \sum_{s \in \mathcal{S}_D} (-1)^{L_s+l_s+\overline{l}_s +\overline{L}_s} q^{-3(L_s-l_s+\overline{l}_s-\overline{L}_s)} q^{D_s-d_s+\overline{d}_s -\overline{D}_s} \\
& = \sum_{s \in \mathcal{S}_D} (-1)^{L_s+l_s+\overline{l}_s +\overline{L}_s} q^{-4(L_s-l_s+\overline{l}_s-\overline{L}_s)} q^{L_s+D_s-(l_s+d_s)+\overline{l}_s+\overline{d}_s-(\overline{L}_s+\overline{D}_s)} \\
& = \sum_{s \in \mathcal{S}_D} (-1)^{L_s+l_s+\overline{l}_s +\overline{L}_s} q^{-4(L_s-l_s+\overline{l}_s-\overline{L}_s)} q^{A_s-a_s+b_s-B_s}.
\end{align*}
Thus, we have $$X_D(\xi)=\sum_{s\in \mathcal{S}_D}\xi^{A_s-a_s+b_s-B_s}.$$
On the other hand, we have $$A_s-a_s+b_s-B_s = A_s+a_s+b_s+B_s-2(B_s+a_s) = c(D)-2(B_s+a_s), \text{and }$$ $$A_s-a_s+b_s-B_s = -(A_s+a_s+b_s+B_s)+2(A_s+b_s) = -c(D)+2(A_s+b_s),$$ where $c(D)$ is the number of crossings of the diagram $D$. Finally, we have
\begin{align*}
 f(\xi) = (-\xi)^{-3w(D)}X_D(\xi) & = (-\xi)^{-3(y(D)-x(D))}\sum_{s\in \mathcal{S}_D}\xi^{c(D)-2B_s-2a_s} \\ & = (-\xi)^{-3(y(D)-x(D))}\xi^{y(D)+x(D)}\sum_{s\in \mathcal{S}_D}i^{-B_s-a_s} \\ & =  (-1)^{-3w(D)}(-\xi)^{-2y(D)+4x(D)}\sum_{s\in \mathcal{S}_D}i^{-B_s-a_s} \\ & = i^{y(D)}\sum_{s\in \mathcal{S}_D}i^{-B_s-a_s} 
\end{align*}
The result now follows for the first claim directly from the definition of the mock-determinant and the fact that $\left|i^{y(D)}\sum_{s\in \mathcal{S}_D}i^{-B_s-a_s}\right| =\left|\sum_{s\in \mathcal{S}_D}i^{-B_s-a_s}\right|$. A similiar argument can be applied to prove the second claim. 
\end{proof}

\begin{exmp} 
In Figure \ref{diagramplusstates}, we depict a checkerboard colorable knot diagram that represents the knot $4.99$ in \cite{G}, together with its five connected states. By Proposition \ref{newprop1}, we get ${\rm mdet}(4.99) = 5$. 

\begin{figure}
\centering
\includegraphics[width=1\linewidth]{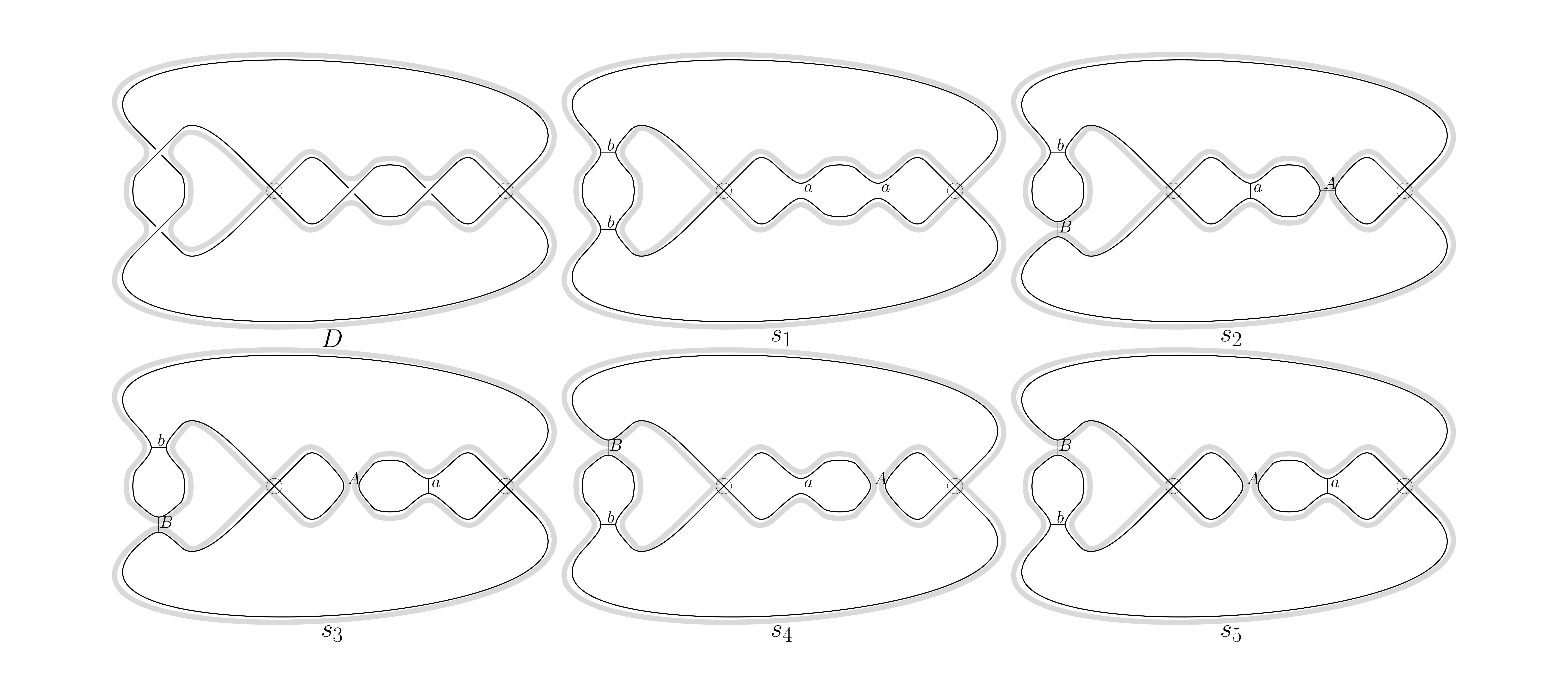}   
\caption{A checkerboard virtual link diagram of the knot $4.99$ and its connected states.}
\label{diagramplusstates}
\end{figure}
\label{exmp1}
\end{exmp}

\section{Mock-alternating Virtual Links}

In this section,  we define the set of mock-alternating virtual links recursively motivated by the definition of the set of quasi-alternating classical links in \cite{OS}.
\begin{defn}\label{def}
The set $\mathcal{M}$ of mock-alternating links is the smallest set of virtual links satisfying the following properties:
\begin{itemize}
  \item The unknot belongs to $\mathcal{M}$.
  \item If $L$ is a virtual link with a diagram $D$ containing a crossing $c$ such that
\begin{enumerate}
\item The two links $L_0$ and $L_1$ represented by the smoothings $D_0$ and $D_1$ of the diagram $D$ at the crossing $c$ respectively belong to $\mathcal{M}$ and satisfy,
\item $\mdet(L_{0}), \mdet(L_{1}) \geq 1$,
\item $\mdet(L) = \mdet(L_{0}) + \mdet(L_{1})$;\\
  then $L$ is in $\mathcal{M}$ and in this case we say that $c$ is a mock-alternating crossing of $D$ and that $D$ is mock-alternating at $c$.
\end{enumerate}
\end{itemize}
\end{defn}

\begin{figure} [h]
\begin{center}
\includegraphics[scale=0.3]{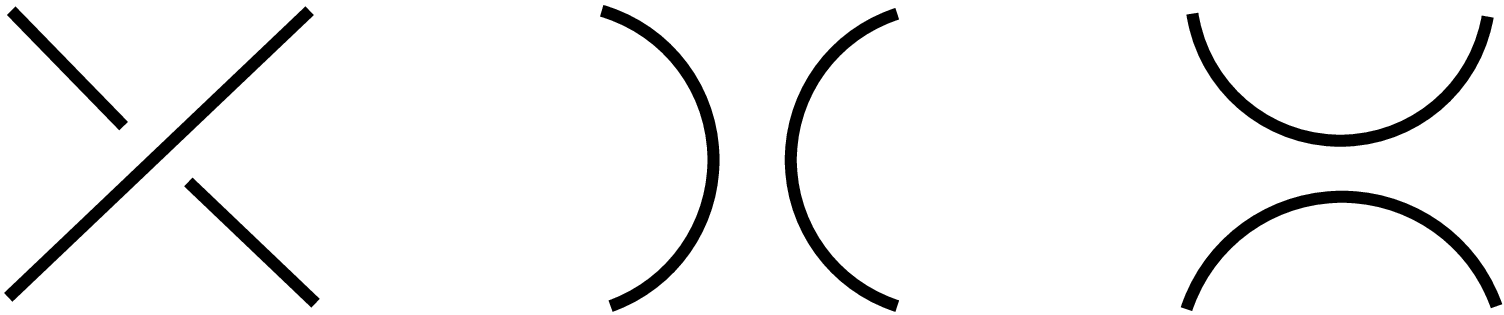} \\
\end{center}
\caption{From the left, a crossing $c$ of a link diagram $D$ and its smoothings $D_0$ and $D_1$ at $c$.}\label{figure}
\end{figure}

The set of quasi-alternating links introduced in \cite{OS} is contained in the set of mock-alternating virtual links. Indeed, the mock-determinant is an extension of the determinant invariant of classical links, and thus the definitions of quasi-alternating and mock-alternating classical links coincide. We give some examples of non-classical mock-alternating virtual links. 
\begin{exmp}
\begin{enumerate}
\item The only virtual link of two components of crossing number one, which is the link 2-1-1 in \cite{BR}, is  mock-alternating at the only crossing of its diagram according to Definition \ref{def}. 
\item The checkerboard colorable knots $3.5$ and $4.99$ in \cite{G} are mock-alternating. In fact, we give a binary tree in Figure \ref{malt} with root the knot $4.99$, and with a sub-tree with root $3.5$, which has as leaves the classical Hopf link $H$ and the unknot $U$ whose mock-determinants are $2$ and $1$, respectively. On the other hand, in Example \ref{exmp1}, we found that $ \mdet(4.99)=5 $  by using Proposition \ref{newprop1}. Similarly, we get $\mdet(3.5) = 3$. Finally, by Definition \ref{def}, the two knots $3.5$ and $4.99$ are mock-alternating. 
\begin{figure}
\centering
\includegraphics[width=1\linewidth]{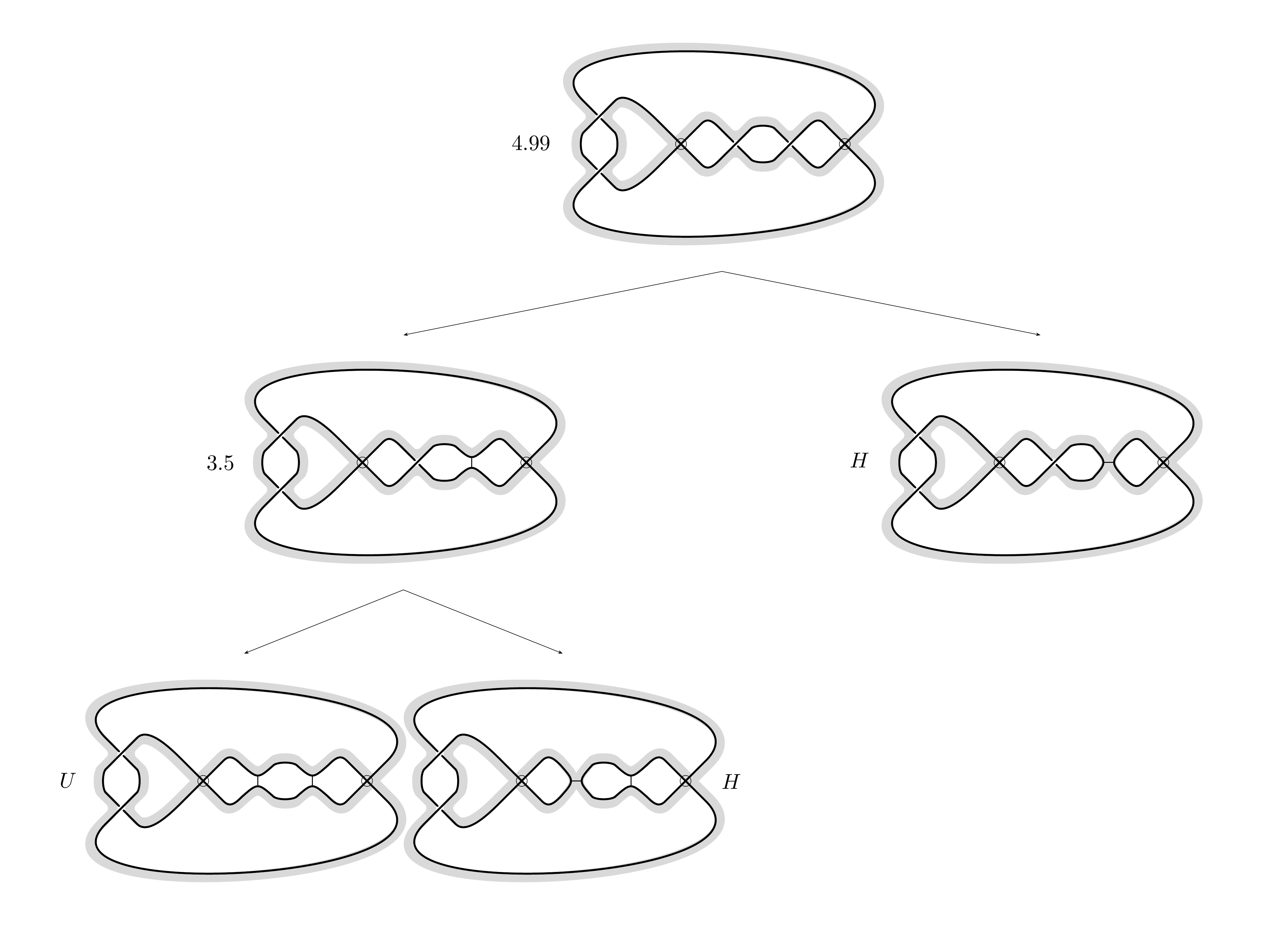}   
\caption{}
\label{malt}
\end{figure}
\end{enumerate}
\label{exmp2}
\end{exmp}
\begin{rem}
The mock-alternating virtual links in Example \ref{exmp2} are not alternating. This implies that the set of mock-alternating links is not contained in the set of alternating virtual links. Furthermore, since the link 2-1-1 in \cite{BR} is mock-alternating and not checkeroard colorable, then the set of mock-alternating links is not contained in the set of checkerboard colorable links. 
\label{newrem}
\end{rem}

According to the recursive definition, it is almost impossible to show that a given link is not mock-alternating. An alternative technique to address such a problem is to introduce an obstruction criterion for a virtual link to be mock-alternating. Then we show that such a link does not pass such an obstruction. To introduce the first obstruction for a virtual link to be mock-alternating, we modify and extend first the notation for a polynomial to be alternating.
\begin{defn}
The polynomial $g(A) = g_{1}(A) + g_{2}(A)$ with $g_{1}(A) \in \mathbb{Z}[A^{4}, A^{-4}]$ and $g_{2}(A) \in \mathbb{Z}[A^{4}, A^{-4}]\cdot A^{2}$ is said to be alternating if both $g_{1}(A)$ and $g_{2}(A)$ are alternating. In particular,  $g_{i}(A)$ is alternating if $ (-1)^{\frac{j-k}{4}}a_{j}a_{k}\geq  0$ whenever $a_{j}$ and $a_{k}$ are coefficients of the two monomials $A^{j}$ and $A^{k}$ in the polynomial $g_{i}(A)$ for $i =1$ or $i=2$.
\end{defn}
\begin{thm}\label{alternating}
Let $L$ be a mock-alternating virtual link, then the Jones-Kauffman polynomial $f_{L}(A)$ is alternating. In particular, the polynomials $\Phi_{L}(A)$ and $\Psi_{L}(A)$ are alternating.
\end{thm}

\begin{proof}
We use induction on the mock-determinant of the link $L$. According to the definition and to the result of Proposition (\ref{split}), $L$ is the unknot whenever $\mdet(L)= 1$ and thus the result follows directly since $f_{A}(L) =1$ in this case. Now we assume that the result holds for any mock-alternating virtual link of mock-determinant less than $\mdet(L)$. In particular, we can deduce that the Jones-Kauffman polynomials for the links $L_{0}$ and $L_{1}$  are alternating. 

To prove the result for the link $L$, we have four cases to consider based on whether the crossing $c$ is positive or negative and $e$ is even or odd. We consider the case when $c$ is positive and $e$ is odd and similar argument can be applied to the other cases. 
In this case using the notion in Remark \ref{simple} and according to Lemma \ref{jonespolynomial}, we have 
\begin{align*}
 f_{L}(A) &   = \Phi_{L}(A) + \Psi_{L}(A) \\ & = -A^{-2}f_{L_{0}}(A)-A^{-6e-4}f_{L_{1}}(A) \\ & = -A^{-2}\left( \Phi_{L_{0}}(A) + \Psi_{L_{0}}(A)\right)-A^{-6e-4}\left( \Phi_{L_{1}}(A) + \Psi_{L_{1}}(A)\right), 
\end{align*}
with 
\begin{align*}
\Phi_{L}(A) =
  \begin{cases}
  -A^{-2}\Psi_{L_{0}}(A)-A^{-6e-4}\Psi_{L_{1}}(A),
     & \text{if $e $ is odd},
     \\
    -A^{-2}\Psi_{L_{0}}(A)-A^{-6e-4}\Phi_{L_{1}}(A), & \text{if $e$ is even},
 \end{cases}
  \end{align*}  
  and   
  \begin{align*}
\Psi_{L}(A) =
  \begin{cases}
  -A^{-2}\Phi_{L_{0}}(A)-A^{-6e-4}\Phi_{L_{1}}(A),
     & \text{if $e $ is odd},
     \\
    -A^{-2}\Phi_{L_{0}}(A)-A^{-6e-4}\Psi_{L_{1}}(A), & \text{if $e$ is even}.
 \end{cases}
  \end{align*}                                                 
In the case when $e$ is odd, we have 
\begin{align*}
 |\Phi_{L}(\xi)| +   |\Psi_{L}(\xi)| & = \left|-\xi^{-2}\Psi_{L_{0}}\left(\xi\right) -\xi^{-6e-4}\Psi_{L_{1}}\left(\xi\right)\right| +  \left|-\xi^{-2}\Phi_{L_{0}}\left(\xi\right) -\xi^{-6e-4}\Phi_{L_{1}}\left(\xi\right)\right| \\ & = \mdet(L) \\ & =  \mdet(L_{0}) +  \mdet(L_{1}) \\ & =   \left|\Phi_{L_{0}}\left(\xi\right)\right|  + \left|\Psi_{L_{0}}\left(\xi\right)\right|  +  \left|\Phi_{L_{1}}\left(\xi\right)\right| + \left|\Psi_{L_{1}}\left(\xi\right)\right|.
\end{align*}

Thus we can conclude that \[ \left|-\xi^{-2}\Psi_{L_{0}}\left(\xi\right) -\xi^{-6e-4}\Psi_{L_{1}}\left(\xi\right)\right| =  \left|\Psi_{L_{0}}\left(\xi\right)\right| + \left|\Psi_{L_{1}}\left(\xi\right)\right|,\] and 
 \[  \left|-\xi^{-2}\Phi_{L_{0}}\left(\xi\right) -\xi^{-6e-4}\Phi_{L_{1}}\left(\xi\right)\right|  = \left|\Phi_{L_{0}}\left(\xi\right)\right|   + \left|\Phi_{L_{1}}\left(\xi\right)\right|,\] since the polynomials $\Psi_{L_{0}}(A), \Psi_{L_{1}}(A), \Phi_{L_{0}}(A),$ and $\Phi_{L_{1}}(A)$ are all alternating.

The polynomials $ \Psi_{L_{0}}^{'}(A) = -A^{-2}\Psi_{L_{0}}(A)$ and $\Psi_{L_{1}}^{'}(A) = -A^{-6e-4}\Psi_{L_{1}}(A)$ are alternating as a result of the fact that  $\Psi_{L_{0}}(A)$ and $\Psi_{L_{1}}(A)$ are alternating. We claim that $a_{i}b_{i} \geq 0$ in the case that $e\equiv 1,5 \pmod 8$ and $a_{i}b_{i} \leq 0$ in the case that $e\equiv 3,7 \pmod 8$ for the coefficients $a_{i}$ and $b_{i}$ of the monomial $A^{i}$ in the two polynomials $\Psi_{L_{0}}(A)$ and $\Psi_{L_{1}}(A)$ respectively. Otherwise, we obtain that $\left|-\xi^{-2}\Psi_{L_{0}}\left(\xi\right) -\xi^{-6e-4}\Psi_{L_{1}}\left(\xi\right)\right| =  |\left|\Psi_{L_{0}}\left(\xi\right)\right| - \left|\Psi_{L_{1}}\left(\xi\right)\right||$ and this contradicts the fact that  $\left|-\xi^{-2}\Psi_{L_{0}}\left(\xi\right) -\xi^{-6e-4}\Psi_{L_{1}}\left(\xi\right)\right| =  \left|\Psi_{L_{0}}\left(\xi\right)\right| + \left|\Psi_{L_{1}}\left(\xi\right)\right|$. Thus we conclude that $\Phi_{L}(A)$  is alternating in this case. A similar argument can be applied for the polynomial $\Psi_{L}(A)$. 
\end{proof}

\begin{coro}\label{obstruction}
If $f_{L}(A)$ is not alternating, then $L$ is not mock-alternating virtual link.
\end{coro}
As an application, we use the obstruction in Corollary \ref{obstruction} to list some of virtual knots that are not mock-alternating.
\renewcommand{\arraystretch}{1.3} 
\begin{table}[H]
\centering
\caption{Virtual knots of four crossings or less in the table of virtual knots in \cite{G} that are not mock-alternating virtual links according to the obstruction in Corollary \ref{obstruction} after the normalization $q^{\frac{1}{2}} = A^{-2}$.}
\vspace{0.5cm}
\begin{tabular}{|c|c|}
\hline
Jones-Kauffman polynomial & Link  \\
\hline
$1$ & 3.1, 3.7, 4.13, 4.19, 4.26, 4.35, 4.42, 4.46, \\ & 4.47,  4.55, 4.56, 4.59,  4.66,
4.67,4.72,  \\ & 4.76, 4.77, 4.85, 4.93, 4.96, 4.97, 4.98,  \\ & 4.102, 4.103, 4.106, 
4.107 \\
\hline
$q^{-5} - 2q^{-4} - 2q^{-\frac{7}{2}} + q^{-3} + 2q^{\frac{-5}{2}} + q^{-2}$ & 4.1, 4.3, 4.7, 4.53, 4.73\\
\hline
$- q^{\frac{-3}{2}} - q^{-1} + q^{\frac{-1}{2}} + 3 + q^{\frac{1}{2}} - q - q^{\frac{3}{2}}$ & 4.2, 4.6, 4.8, 4.12, 4.17, 4.32, 4.51,  \\ & 4.58, 4.71, 4.75 \\
\hline
$q^{-5} - 2q^{-4} - 2q^{\frac{-7}{2}} + q^{-3} + 2q^{\frac{-5}{2}} + q^{-2}$ & 4.1, 4.3, 4.7, 4.53, 4.73 \\
\hline
$- q^{-4} - q^{\frac{-7}{2}} + q^{-3} + q^{\frac{-5}{2}} + q^{-2}$ &4.9, 4.15, 4.29, 4.37, 4.48, 4.61, 4.69, 4.78 \\
\hline 
$- q^{-3} - q^{\frac{-5}{2}} + q^{-2} + 2q^{\frac{-3}{2}} + q^{-1} - q^{\frac{-1}{2}}$ & 4.10, 4.16, 4.21, 4.24, 4.31, 4.36, 4.50,  \\ & 4.57, 4.68, 4.70
 \\
\hline
\end{tabular}
\vspace{0.2cm}
\label{tab1}
\end{table}
\renewcommand{\arraystretch}{1.0} 
\begin{rem}
 The virtual alternating link in the only remark of \cite{K} is not mock-alternating since the splitting of the polynomial $f_{L}(A)$ gives rise to a nonalternating polynomial. This implies that the set of alternating virtual links is not contained in the set of mock-alternating virtual links. This can further be confirmed by the fact that the alternating virtual knots $3.7$, $4.106$, and $4.107$ are not mock-alternating as one can check using Table \ref{tab1}. In other words; the set of mock-alternating virtual links is not a generalization of alternating virtual links.
\end{rem}
In the last part of this section, we use the twisting technique to generate an infinite family of mock-alternating virtual links from a given mock-alternating virtual link. This technique was described for the first time in \cite{CK} using rational tangles and later generalized in \cite{QCQ} to a product of rational tangles and to non-algebraic alternating tangles in \cite{CKa}. The main idea of this technique lies on replacing the crossing $c$ where the virtual link $L$ is mock-alternating by some tangle. Before we achieve such a goal, we recall and adopt the notion of extending a crossing by a rational tangle in \cite{CK} to virtual links as follows:
\begin{defn}
We say that a rational tangle of associated continued fraction $[a_{1},a_{2},\ldots,a_{n}]$ extends the crossing $c$ if it contains the crossing $c$ and we have $\epsilon(c) a_{i} \geq 1$ for all $i$ where the $\epsilon(c) = \pm 1$ based on whether the crossing is of type I or type II, respectively. In particular, this rational tangle is an alternating rational tangle. We refer the reader to \cite{KaLo} for further details on rational tangles and their continued fraction. 

\end{defn}

\begin{thm}\label{new2}
If $D$ is a mock-alternating virtual link diagram at some crossing $c$, then the link diagram  $D^{c\leftarrow r}$ obtained from $D$ by replacing the crossing $c$ by a rational tangle $r$ that extends $c$ is mock-alternating at each crossing of $r$.
\end{thm}

\begin{proof}
  We may assume that $\epsilon(c) =1$  by taking the mirror image if it is needed. We prove the result by induction on the length of the continued fraction of the given rational tangle. If the length is one, then we have an integer tangle in such a case. Now we prove this case by induction on the number of crossings of this integer tangle. Both such inductions depend on proving the result for the link diagram $D^{c\leftarrow 2}$ obtained by replacing the crossing $c$ by an integer tangle of two crossings that extends $c$. In other words, the link diagram $D^{c\leftarrow 2}$ is obtained from the link diagram $D$ by replacing the crossing $c$ by two vertical or horizontal positive half-twists. Now we consider the first case and second case can be treated in a similar manner. If we smooth the second crossing of such an integer tangle, we obtain the link diagrams $D_{0}^{c\leftarrow 2}$ and $D_{1}^{c\leftarrow 2}$ which are exactly the mock-alternating diagrams $D$ and $D_{1}$ at $c$ respectively.
 
We denote $L$, $L_0$, $L_1$ and $L^{c\leftarrow r}$ the links represented by the diagrams $D$, $D_0$, $D_1$ and $D^{c\leftarrow r}$ respectively.
To prove the result, it is left to prove that $\mdet(L^{c\leftarrow 2}) = \mdet(L_{0}^{c\leftarrow 2}) + \mdet(L_{1}^{c\leftarrow 2}) = \mdet(L) + \mdet(L_{1}) = \mdet(L_{0}) + 2\mdet(L_{1})$. As a consequence of Lemma \ref{jonespolynomial}, we obtain 
\begin{align*}
 f_{L^{c\leftarrow 2}}(A) & = -A^{-2}f_{L_{0}^{c\leftarrow 2}}(A)-A^{-6e'-4}f_{L_{1}^{c\leftarrow 2}}(A) \\
& = -A^{-2}f_{L}(A)-A^{-6e'-4}f_{L_{1}}(A) \\
& = -A^{-2}\left(\Phi_{L}(A) + \Psi_{L}(A)\right)-A^{-6e'-4}\left(\Phi_{L_{1}}(A) + \Psi_{L_{1}}(A)\right) \\ &
=  -A^{-2}\left(-A^{-2}\Psi_{L_{0}}(A)-A^{-6e-4}\Psi_{L_{1}}(A) -A^{-2}\Phi_{L_{0}}(A)-A^{-6e-4}\Phi_{L_{1}}(A)\right) \\ & \ \  \ -A^{-6e'-4}\left(\Phi_{L_{1}}(A) + \Psi_{L_{1}}(A)\right). 
\end{align*}

Note that \begin{align*}
\Phi_{L^{c\leftarrow 2}}(A) & = A^{-4}\Phi_{L_{0}}(A)  +A^{-6e-6}\Phi_{L_{1}}(A)  - A^{-6e'-4}\Phi_{L_{1}}(A) \\ & = A^{-4}\Phi_{L_{0}}(A)  +A^{-6e-6}(1-A^{-4})\Phi_{L_{1}}(A) \\ & = -A^{-2}\left(-A^{-2}\Phi_{L_{0}}(A)  - A^{-6e-4}(1-A^{-4})\Phi_{L_{1}}(A)\right) .
\end{align*}
 and  
\begin{align*}
\Psi_{L^{c\leftarrow 2}}(A) & = A^{-4}\Psi_{L_{0}}(A)  +A^{-6e-6}\Psi_{L_{1}}(A)  - A^{-6e'-4}\Psi_{L_{1}}(A) \\ & = A^{-4}\Psi_{L_{0}}(A)  +A^{-6e-6}(1-A^{-4})\Psi_{L_{1}}(A) \\ &= -A^{-2}\left(-A^{-2}\Psi_{L_{0}}(A)  - A^{-6e-4}(1-A^{-4})\Psi_{L_{1}}(A)\right).
\end{align*}

The last two equations hold since $e'= e+1$ when the crossing $c$ in the link diagram $D$ is positive and the corresponding crossings in $D^{c\leftarrow 2}$ are also positive and $e$ is odd. Now we know that $\mdet(L) = \mdet(L_{0}) + \mdet(L_{1})$ which is equivalent to 
\begin{align*}
 |\Phi_{L}(\xi)| +   |\Psi_{L}(\xi)| & = \left|-\xi^{-2}\Psi_{L_{0}}\left(\xi\right) -\xi^{-6e-4}\Psi_{L_{1}}\left(\xi\right)\right| +  \left|-\xi^{-2}\Phi_{L_{0}}\left(\xi\right) -\xi^{-6e-4}\Phi_{L_{1}}\left(\xi\right)\right| \\ & = \mdet(L) \\ & =  \mdet(L_{0}) +  \mdet(L_{1}) \\ & =   \left|\Phi_{L_{0}}\left(\xi\right)\right|  + \left|\Psi_{L_{0}}\left(\xi\right)\right|  +  \left|\Phi_{L_{1}}\left(\xi\right)\right| + \left|\Psi_{L_{1}}\left(\xi\right)\right|.
\end{align*}

This implies that   
\begin{align*}
  \left|-\xi^{-2}\Psi_{L_{0}}\left(\xi\right) -\xi^{-6e-4}\Psi_{L_{1}}\left(\xi\right)\right| +  \left|-\xi^{-2}\Phi_{L_{0}}\left(\xi\right) -\xi^{-6e-4}\Phi_{L_{1}}\left(\xi\right)\right| = \\   \left|\Phi_{L_{0}}\left(\xi\right)\right|  + \left|\Psi_{L_{0}}\left(\xi\right)\right|  +  \left|\Phi_{L_{1}}\left(\xi\right)\right| & + \left|\Psi_{L_{1}}\left(\xi\right)\right|.
\end{align*}

Thus we conclude that  \[ \left|-\xi^{-2}\Psi_{L_{0}}\left(\xi\right) -\xi^{-6e-4}\Psi_{L_{1}}\left(\xi\right)\right| =  \left|\Psi_{L_{0}}\left(\xi\right)\right| + \left|\Psi_{L_{1}}\left(\xi\right)\right|,\] and 
 \[  \left|-\xi^{-2}\Phi_{L_{0}}\left(\xi\right) -\xi^{-6e-4}\Phi_{L_{1}}\left(\xi\right)\right|  = \left|\Phi_{L_{0}}\left(\xi\right)\right|   + \left|\Phi_{L_{1}}\left(\xi\right)\right|,\] since the polynomials $\Psi_{L_{0}}(A), \Psi_{L_{1}}(A), \Phi_{L_{0}}(A),$ and $\Phi_{L_{1}}(A)$ are all alternating.

Using similar argument as in the proof of Theorem \ref{alternating}, we can assume that $a_{i}b_{i} \geq 0$ in the case that $e\equiv 1,5 \pmod 8$ and $a_{i}b_{i} \leq 0$ in the case that $e\equiv 3,7 \pmod 8$ for the coefficients $a_{i}$ and $b_{i}$ of the monomial $A^{i}$ in the two polynomials $\Psi_{L_{0}}(A)$ and $\Psi_{L_{1}}(A)$ respectively. 
Clearly, this would imply that  the coefficients $a_{i}$ and $b_{i}$ of any monomial $A^{i}$ in the two polynomials $ -A^{-2}\Phi_{L_{0}}(A)$ and $ -A^{-6e-4}(1-A^{-4})\Phi_{L_{1}}(A)$ respectively satisfy $a_{i}b_{j}\geq0 $. Thus we conclude that 
\begin{align*}
\left|\Phi_{L^{c\leftarrow 2}}(\xi)\right| & =  \left|-\xi^{-2}\left(-\xi^{-2}\Phi_{L_{0}}(\xi)  - \xi^{-6e-4}(1-\xi^{-4})\Phi_{L_{1}}(\xi)\right)\right| \\ & = \left |-\xi^{-2}\Phi_{L_{0}}(\xi)  - 2\xi^{-6e-4}\Phi_{L_{1}}(\xi)\right| \\ & = \left|\Phi_{L_{0}}(\xi)\right| + 2 \left|\Phi_{L_{1}}(\xi)\right|.
 \end{align*}

 A similar argument can be used to prove that
 \begin{align*}
\left|\Psi_{L^{c\leftarrow 2}}(\xi)\right| & =  \left|-\xi^{-2}\left(-\xi^{-2}\Psi_{L_{0}}(\xi)  - \xi^{-6e-4}(1-\xi^{-4})\Psi_{L_{1}}(\xi)\right)\right| \\ & =  \left|-\xi^{-2}\Psi_{L_{0}}(\xi)  - 2\xi^{-6e-4}\Psi_{L_{1}}(\xi)\right| \\ & = \left|\Psi_{L_{0}}(\xi)\right| + 2 \left|\Psi_{L_{1}}(\xi)\right|.
\end{align*}

This proves the required and thus we conclude that $L^{c\leftarrow 2}$ is mock-alternating. Now we  assume that the result holds for any integer tangle that replaces the crossing $c$ by $n-1$ half-twists and it is mock-alternating at the last crossing of such integer tangle. Now we incorporate the previous argument for $L^{c\leftarrow 2}$ to show the result when the crossing $c$ is replaced by an integer tangle of $n$ half-twists. At the end, the result holds for any rational tangle that extends the crossing $c$ using the induction hypothesis and the fact that any rational tangle of continued fraction $[a_{1},a_{2}, \ldots, a_{n}]$ can be expressed by the continued fraction $[a_{1},a_{2}, \ldots, a_{n}-1,1]$. This can be done by replacing the last crossing in the latter continued fraction by any integer tangle using the first induction.

\end{proof}

By applying a similar argument used in the proof of the above theorem, we can conclude the following corollary:
\begin{coro}
Let $L$ be a mock-alternating virtual link. Let $D$ be a virtual link diagram of $L$ which is mock-alternating at the crossing $c$, then
\[ \mdet(L^{c\leftarrow r}) =
  \begin{cases}
  \alpha \mdet(L_{0}) + \beta \mdet(L_{1}),
     & \text{if $\epsilon(c) = -1 $},
     \\
    {\alpha \mdet(L_{1}) + \beta \mdet(L_{0}}), & \text{if $\epsilon(c) = 1 $},
  \end{cases}
\]
where $r$ is a rational tangle of slope $\frac{\beta}{\alpha}$ that extends
$c$.
\end{coro}

We like to enclose this section with the following two questions motivated by the conjectures for the set of quasi-alternating links in \cite{G2,QQJ}, respectively. 

\begin{ques}\label{main2}
Is there only finitely many mock-alternating virtual links of given mock-determinant?
\end{ques}
\begin{ques}\label{main1}
Does the mock-determinant represent an upper bound of the crossing number for any mock-alternating or alternating virtual link?
\end{ques}

\subsection*{Acknowledgment}
The authors acknowledge that AI has been used for the sole purpose of typesetting support, grammatical correction, and paraphrasing  to enhance readability and improve the manuscript's language. After this usage, the authors reviewed, edited and approved the final content, and took full responsibility for the scientific integrity and accuracy of the text.

\end{document}